\documentclass[12pt,a4paper]{amsart}
\usepackage{mathrsfs}
\usepackage[active]{srcltx}
\usepackage{enumerate}
\usepackage{amsmath,amssymb,xspace,amsthm}

\usepackage{cite}
\usepackage{graphicx}
\usepackage{amscd}
\usepackage{amsfonts}
\usepackage{color}
\usepackage[mathscr]{eucal}
\usepackage{latexsym}

\usepackage{ifpdf}
\ifpdf
  \usepackage[colorlinks,final,backref=page,hyperindex]{hyperref}
\else
  \usepackage[colorlinks,final,backref=page,hyperindex,hypertex]{hyperref}
\fi

\begin{document}

\theoremstyle{definition}
\newtheorem{definition}{Definition}[section]
\newtheorem{theorem}[definition]{Theorem}
\newtheorem{example}[definition]{Example}
\newtheorem{remark}[definition]{Remark}
\newtheorem{lemma}[definition]{Lemma}
\newtheorem{proposition}[definition]{Proposition}
\newtheorem{corollary}[definition]{Corollary}
\newtheorem{conjecture}[definition]{Conjecture}
\newtheorem{case}{Case}
\newtheorem{subcase}{Subcase}[case]
\newtheorem{claim}{Claim}
\numberwithin{equation}{section}

\def\W{\mathcal{W}}
\def\V{\Vir}

\newcommand{\C}{\mathbb{C}}
\newcommand{\F}{\mathbb{F}}
\newcommand{\N}{\mathbb{N}}
\newcommand{\Z}{\mathbb{Z}}

\def\ad{\mathrm{ad\ }}
\def\diag{\mathrm{diag}}
\def\rank{\mathrm{rank }}
\def\spn{\mathrm{span}}
\def\Ind{\mathrm{Ind}}
\def\tr{\mathrm{tr}}
\def\Der{\mathrm{Der}}
\def\id{\mathrm{id}}

\def\bp{\mathbf{p}}
\def\bl{\mathbf{l}}
\def\fa{\mathfrak{a}}
\def\fe{\mathfrak{e}}
\def\fg{\mathfrak{g}}
\def\gl{\mathfrak{gl}}
\def\te{\tilde{e}}
\def\tf{\tilde{f}}
\def\tv{\widetilde{V}}
\def\tmu{\widetilde{\mu}}

\def\LL{\mathcal{L}}
\def\a{\alpha}
\def\b{\beta}
\def\d{\delta}
\def\D{\Delta}
\def\l{\lambda}

\def\p{\partial}
\def\ra{\rightarrow}

\def\gen#1{\langle #1 \rangle}

\def\HVir{\mathcal{L}}
\def\Vir{\mathrm{Vir}}

\def\BB{\mathcal{B}}

\def\li#1{\textcolor{red}{#1}}

\renewcommand{\labelenumi}{(\arabic{enumi}).}
\renewcommand{\labelenumii}{(\roman{enumii})}

\title[$\lambda$-Differential operators and modules for Virasoro algebra]{$\lambda$-Differential operators and $\lambda$-differential modules for the Virasoro algebra}

\author{Xuewen Liu}
\address{School of Mathematics and Statistics, Lanzhou
University, Lanzhou 730000, Gansu, China
and School of Mathematics and Statistics, Zhengzhou
University, Zhengzhou 450001, Henan, China}
\email{liuxw@zzu.edu.cn}

\author{Li Guo}
\address{Department of Mathematics and Computer Science, Rutgers
University, Newark, NJ 07102, USA}
\email{liguo@rutgers.edu}

\author{Xiangqian Guo}
\address{School of Mathematics and Statistics, Zhengzhou
University, Zhengzhou 450001, Henan, China}
\email{guoxq@zzu.edu.cn}


\begin{abstract} The concept of $\l$-differential operators is a natural generalization of
differential operators and difference operators. In this paper, we
determine the $\l$-differential Lie algebraic structure on the Witt
algebra and the Virasoro algebra for invertible $\l$. Then we
consider several families of modules over the Virasoro algebra with explicit module
actions and determine the $\l$-differential module structures on them.
\end{abstract}

\subjclass[2010]{17B10,17B68,15A04,17B40,15A69,17B65} 

\keywords{differential Lie algebra, differential module, Witt algebra, Virasoro algebra \\
Published as {\em Linear and Multilinear Algebra} {\bf 67} (7) (2019), 1308-1324
}

\maketitle

\vspace{-1cm}

\tableofcontents

\vspace{-1cm}

\section{Introduction}

The Virasoro algebra $\Vir$ is an infinite dimensional Lie algebra
over the complex numbers $\C$, with basis $\{L_n,C\,\, |\,\, n \in
\Z\}$ and defining relations
$$
[L_{m},L_{n}]=(n-m)L_{n+m}+\delta_{n,-m}\frac{m^{3}-m}{12}C, \quad
m,n \in \Z,
$$
$$
[C, L_{m}]=0, \quad m \in \Z.
$$
The quotient algebra of $\Vir$ modulo the center $\C C$ is just the (rank-$1$)
Witt algebra. The algebra $\Vir$ is one of the most important Lie
algebras both in mathematics and in mathematical physics. See for
example \cite{IK, KR} and references therein. In particular, it has
been widely used in quantum physics \cite{GO}, conformal field
theory \cite{DMS}, Kac-Moody algebras \cite{K,MoP} and vertex operator
algebras \cite{DMZ,FZ}.

The representation theory of the
Virasoro algebra has attracted a lot of attention from
mathematicians and physicists. There are many classical results on
the theory of Harish-Chandra modules (weight modules with finite
dimensional weight spaces, including highest weight modules and
modules of intermediate series), see \cite{KR, IK, FF, M}. Recently non-Harish-Chandra modules
are investigated extensively. For example, weight modules with infinite dimensional weight spaces
were studied in \cite{CGZ, CM, LLZ} and several new families of non-weight
modules (including Whittaker modules and their generalizations) were constructed and studied in \cite{CG, GLZ, GWL, LGZ, LZ, MW, MZ, OW, TZ1, TZ2}.

The study of differential associative algebras began with the algebraic approach of Ritt~\cite{Ri} to differential equations. Through the work of Kolchin and many other mathematicians, it has expanded into a vast area of research, including differential Galois theory and differential algebraic geometry, with broad applications in mathematics and physics~\cite{Kol,Sv}.
The algebraic abstraction of difference equations led to the concept of a difference algebra which has developed largely in parallel to the differential algebra~\cite{Co,Le}.
As a natural generalization of both the differential operator and
the difference operator, the concept of an $\l$-differential
operator, for any scalar $\lambda$, was introduced in~\cite{GK}.

Differential graded Lie algebras (dg-Lie algebras) have been studied for some time in connection with simplicial Lie algebras and $L_\infty$-algebras~\cite{Qu,KM}. Differential structures on Lie algebras have also been extensively studied in the contexts of ring of differential operators, $D$-modules, especially Weyl algebras~\cite{Bj,Bo}. These provides strong motivation for a systematic study of differential Lie algebras. In a recent paper~\cite{Po}, a Poincar\'e-Birkhoff-Witt theorem for differential Lie algebras was established.

In order to study differential Lie algebra and difference Lie algebra at the same time, in the present paper, we define the $\l$-differential operator for Lie algebras, obtaining $\l$-differential Lie algebras. For their representations, we also introduce $\l$-differential modules for the Lie algebra carrying a $\l$-differential operator.
Before a general study of these concepts, we focus on the concrete yet important examples of the Virasoro algebra and its modules.
More precisely, we study $\l$-differential operators on the Virasoro algebra (and on the Witt algebra). Then we study the corresponding representations by determining $\l$-differential modules on these $\l$-differential algebras.
In Section~\ref{sec:diffalg}, we determine the $\l$-differential Lie algebraic
structures on the Witt algebra and the Virasoro algebra. In Section~\ref{sec:diffmod}, we determine the $\l$-differential module structures for several classes of Virasoro modules with explicit module structures, including Verma modules, modules of intermediate series and several non-weight modules with explicit module structures.

Throughout this paper, we denote by $\Z$, $\Z_+$, $\N$ and $\C$ the
sets of all integers, nonnegative integers, positive integers and
complex numbers, respectively. All spaces and algebras are over $\C$
and all linear maps are $\C$-linear.

\section{$\l$-differential operators on the Virasoro algebra}
\label{sec:diffalg}

\begin{definition}\label{def-l-algebra}  Let $\LL$ be a Lie algebra and $\l\in\C$. A \textbf{$\lambda$-differential operator} on $\LL$ is a linear map
$d: \LL\rightarrow \LL$ that satisfies
\begin{equation}\label{def}
d[x,y]=[d(x),y]+[x,d(y)]+\lambda[d(x),d(y)],\ \forall\
x,y\in \LL.
\end{equation}
Then $(\LL, d)$ is called a \textbf{$\lambda$-differential Lie algebra}.
\end{definition}

Note that a $0$-differential operator is nothing but a derivation on
$L$. The $\l$-differential operator on an associative algebra~\cite{GK} is modeled on the operator $D_\l$ defined by $D_\l(f)(x)= (f(x+\l)-f(x))/\l$ in the definition of the derivative. Thus the usual derivation is the limit of the $\l$-differential operator when $\l$ approaches zero.

For any $\l\in\C^*=\C\setminus\{0\}$, $d$ is a $\l$-differential operator if and
only if $\l d$ is a $1$-differential operator. Therefore, to
determine all $\l$-differential operators with $\l\neq0$, we need
only to determine all $1$-differential operators. We next relate $1$-differential operator to a better known operator.

\begin{lemma}\label{1-diff-homo} Let $\LL$ be a Lie algebra. A linear map $d: \LL\rightarrow \LL$ is a $1$-differential operator if and
only if $d+\id_\LL$ is a Lie algebra homomorphism.
\end{lemma}

\begin{proof}
When $\l=1$, Eq.~(2.1) is equivalent to
$$d[x, y]+[x, y]=[d(x), y]+[x,d(y)]+[d(x), d(y)]+[x, y]=[d(x)+x, d(y)+y].$$
The lemma follows.
\end{proof}

An associative algebra equipped with an injective multiplicative linear operator is called a difference algebra~\cite{Co,Le}. In this context, the current paper can be regarded as a study of difference Lie algebras and their representations.

Let $\C[t^{\pm1}]$ be the associative algebra of Laurent polynomials in $t$. The Witt algebra $W$ is the derivation Lie algebra of $\C[t^{\pm1}]$, i.e.,
$$W=\Der(\C[t^{\pm1}])=\left\{f(t)\frac{d}{dt}\mid f(t)\in\C[t^{\pm1}]\right\},$$ and the Virasoro algebra $\Vir$ is just the universal central extension of $W$.
More precisely, the Lie bracket of the Witt algebra can be written as
$$\Big[t^{i+1}\frac{d}{dt},\ t^{j+1}\frac{d}{dt}\big]=(j-i)t^{i+j+1}\frac{d}{dt},\ \forall\ i,j\in\Z.$$
We have the following natural epimorphism of Lie algebras $\phi:\ \Vir\rightarrow W$ given by
$$\phi(C)=0,\quad \phi(L_i)=t^{i+1}\frac{d}{dt},\ \forall\ i\in\Z,$$
which realizes $\Vir$ as a universal central extension of $W$.

There are natural $\Z$-gradations on both $\Vir$ and $W$ given by
\begin{equation*}
\Vir=\bigoplus_{n\in\Z}\Vir_n\quad\text{and}\quad W=\bigoplus_{n\in\Z}W_n,
\end{equation*}
where $W_n=\C L_n$ and $\Vir_n=\C L_n\oplus\delta_{n,0}\C C$ for any $n\in \Z$.
These $\Z$-gradations are obviously inherited from the usual $\Z$-gradation of the Laurent polynomial
algebra $\C[t^{\pm1}]$ defined by $\deg(t)=1$, and the homomorphism $\phi$ in the last paragraph respects the gradations.
A $\l$-differential operator $d$ on $\Vir$ or $W$ is called a
\textbf{homogeneous $\l$-differential operator of degree $k$} if it satisfies
\begin{equation*}
d(\Vir_m)\subseteq \Vir_{m+k}\quad\text{or}\quad d(W_m)\subseteq W_{m+k},\ \forall\ m\in\Z.
\end{equation*}
For any $n\in\Z^*$ and $a\in\C^*$, we can define
the homomorphism $\varphi_n$ and $\tau_a$ as follows:
$$\varphi_n: \Vir\rightarrow \Vir,\ \varphi_n(L_i)=\frac{1}{n}(L_{ni}-\delta_{i,0}\frac{n^2-1}{24}C), \varphi_n(C)=nC$$
and
$$\tau_a: \Vir\rightarrow \Vir,\ \tau_a(L_i)=a^i L_i, \tau_a(C)=C.$$
It is clear that $\tau_a$ and $\varphi_{\pm1}$ are automorphisms on $\Vir$.
Furthermore, $\varphi_m\varphi_n=\varphi_{mn}$, $\tau_a\tau_b=\tau_{ab}$ and $\tau_a\varphi_n=\varphi_n\tau_{a^n}$.
It was shown in\cite{Z} that any nonzero homomorphism on the Virasoro algebra $\Vir$ is of the form
$\varphi_n\tau_a$ for some $n\in\Z^*$ and $a\in\C^*$. Thus by Lemma \ref{1-diff-homo}, we see that
$$d_{0,0}:=-\id_{\Vir}\quad \text{and} \quad d_{n,a}:=\varphi_n\tau_a-\id_{\Vir}, n\in\Z^*, a\in\C^*,$$
exhaust all $1$-differential operator, corresponding to the zero and nonzero homomorphisms on $\Vir$ respectively. Thus we have proved

\begin{theorem}\label{1-diff-Vir} Any $1$-differential operator on Virasoro algebra $\Vir$ is
of the form $d_{0,0}$ or $d_{n,a}$ for some $n\in\Z^*$ and $a\in\C^*$. For $\l\in\C^*$, any $\l$-differential operator on Virasoro algebra $\Vir$ is one
of $\l^{-1}d_{0,0}$ or $\l^{-1}d_{n,a}$ for some $n\in\Z^*$ and $a\in\C^*$.
\end{theorem}

\section{$\l$-Differential modules on the Virasoro algebra}
\label{sec:diffmod}

The representation theory of the Virasoro algebra has many applications in mathematical physics. In particular, the Harish-Chandra module theory has been studied extensively and intensively. Recently, many non-Harish-Chandra modules (including non-weight modules and weight modules with infinite dimensional weight spaces) have been constructed and studied by several authors, applying various methods and techniques; see, for example,\cite{CGZ, CM, LLZ, CG, LGZ, LZ, MW, MZ, OW}. In this section, we will determine the $\l$-differential module structures for most of these modules. This will provide several families of $\l$-differential modules over the Virasoro algebra, which may help us understand the concepts of $\l$-differential Lie algebras and  $\l$-differential modules.

Based on Theorem~\ref{1-diff-Vir}, we only need to consider differential modules over $(\Vir,d_{0,0})$ and over $(\Vir,d_{n,a})$ where $n\in \Z^*$ and $a\in \C^*$. The first case is easy to study and will be carried out in Section~\ref{ss:00} together with some general setup. The second case is considered in Section~\ref{ss:na} which is further divided into several types of modules.

\subsection{General setup and the case of $(\Vir,d_{0,0})$}
\label{ss:00}

Let us first recall the definition of a $\l$-differential module.

\begin{definition} Let $(\LL,d)$ be a $\l$-differential Lie algebra. An $\LL$-module $V$ is called a \textbf{$\l$-differential module} for $(\LL, d)$ if there is a linear
map $\delta: V\rightarrow V$ such that
\begin{equation}\label{def-l-mod}
\delta (xv)=d(x)v +x \delta(v) + \lambda d(x)
\delta(v), \forall\ x\in \LL, v\in V.
\end{equation}
We denote this module by $(V,\delta)$.
\end{definition}
When $\lambda=0$, such a $\delta$ is also called a connection in differential geometry.

\begin{lemma}\label{1-diff-mod} Let $(\LL,d)$ be a $1$-differential Lie algebra and $V$ be an $\LL$-module
with a linear map $\delta: V\rightarrow V$. Then $(V,\delta)$ is a $1$-differential $(\LL,d)$-module if and only if the following equation holds:
\begin{equation}\label{Delta-D}
\Delta(xv)=D(x)\Delta(v), \ \forall\ x\in\LL, v\in V,
\end{equation}
where $D=d+\id_\LL$ and $\D=\d+\id_V$.
\end{lemma}
\begin{proof}
When $\l=1$, Eq.~\eqref{def-l-mod} is equivalent to
$$\d(xv)+xv=d(x)v+x\d(v)+d(x)\d(v)+xv=(x+d(x))(v+\d(v)).$$
The lemma follows.
\end{proof}

Note that in all cases, $\delta=-\id_{V}$ always
define a $1$-differential module structure on an $\LL$-module $V$.
We call this a \textbf{trivial $1$-differential module structure}.

As in the case of $\l$-differential operators, we have the following simple relation between $\l$-differential modules and $1$-differential modules.

\begin{corollary}\label{l-diff-mod} Suppose $\l\neq 0$. Let $(\LL,d)$ be a $\l$-differential Lie algebra and $V$ be an $\LL$-module with a linear map $\delta: V\rightarrow V$. Then $(V,\delta)$ is a $\l$-differential $(\LL,d)$-module if and only if the following equation holds:
\begin{equation}\label{Delta-D2}
\Delta_\l(xv)=D_\l(x)\Delta_\l(v), \ \forall\ x\in\LL, v\in V,
\end{equation}
where $D_\l=\l d+\id_\LL$ and $\D_\l=\l\d+\id_V$.
\end{corollary}

Now we are ready to determine the $1$-differential module structures for the Virasoro modules. According to Theorem~\ref{1-diff-Vir}, we only need to consider this question for the $1$-differential algebra
$(\Vir, d_{0,0})$ or $(\Vir, d_{n,a})$ for some $n\in\Z^*=\Z\setminus{0}$ and $a\in\C^*$.

We first dispose the easy case of $(\Vir, d_{0,0})$ and defer the other case to Section~\ref{ss:na}. For any $\Vir$-module $V$, denote by $\Vir(V)$ the space spanned by all elements of the form $xv, x\in\Vir, v\in V$. Take a subspace $V'\subseteq V$ such that $V=V'\oplus \Vir(V)$ and a linear map $\d': V'\rightarrow V$. Define a linear map $\d: V\rightarrow V$ as follows:
$\d\big|_{V'}=\d'$ and $\d(v)=-v$ for any $v\in \Vir(V)$. Then $(V,\d)$ is a $1$-differential $(\Vir,d_{0,0})$-module.

Moreover, any $1$-differential
$(\Vir,d_{0,0})$-module can be obtained in this way.
Indeed, let $(V,\d)$ be a $1$-differential $(\Vir,d_{0,0})$-module. Noticing $d_{0,0}=-\id_\Vir$, we see that Eq.~\eqref{Delta-D} is equivalent to
$$\delta(xv)=-xv,\ \forall\ x\in\Vir, v\in V,$$
which has nothing to do with the values of $\d$ on a complementary space $V'$ of $\Vir(V)$ in $V$.
That is, $d_{0,0}$ is defined as in the previous paragraph.

In summary, we have the description of $1$-differential $(\Vir,d_{0,0})$-modules:
\begin{theorem}
All $1$-differential $(\Vir,d_{0,0})$-modules are precisely the pairs $(V,\delta)$ where $V$ is a
$\Vir$-module and $\delta:V\to V$ is a linear map such that $\delta(xv)=-xv$ for $x\in\Vir, v\in V$.
\end{theorem}

\subsection{Differential modules over $(\Vir,d_{n,a})$}
\label{ss:na}

We next consider the case of differential modules over $(\Vir, d_{n,a})$ for fixed $n\in\Z^*=\Z\setminus{0}$ and $a\in\C^*$. For notational simplicity, we denote $d:=d_{n,a}$ and $D:=d+\id_{\Vir}$. We will consider differential structures for several types of modules over $\Vir$, including the classic Verma modules and modules of intermediate series, and some recently introduced non-Harish-Chandra modules.

\subsubsection{Differential structures on the Verma modules}
\label{sss:verma}

We first consider the Verma modules. For any Lie algebra $\LL$, denote its universal enveloping algebra by $U(\LL)$. Let $\Vir_{\pm}$ be the subalgebras of $\Vir$ spanned by $L_{\pm i}, i\in\N$, respectively. For any $h,c\in\C$, define a $1$-dimensional module $\C v_0$ over the subalgebra $\Vir_+\oplus\C L_0\oplus\C C$ by $L_0v_0=h v_0$, $C v_0=cv_0$ and $L_i v_0=0$ for all $i> 0$. Then we get the induced $\Vir$-module $$M(h,c)=\Ind_{U(\Vir_+\oplus \C L_0\oplus \C C)}^{U(\Vir)}\C v_0,$$
called the \textbf{Verma module} of highest weight $(h,c)$. Denote $M=M(h,c)$ for short.
By the Poincar\'e-Birkhoff-Witt Theorem, we see that
$$M=\bigoplus_{i_1\geq i_2\cdots\geq i_m\geq 1}\C L_{-i_1}L_{-i_2}\cdots L_{-i_m}v_0.$$
Moreover, $M$ has a weight decomposition
$$M=\bigoplus_{i\in\Z}M_{-i},\ \ \text{where}\ M_{-i}=\{v\in M\ |\ L_0v=(h-i)v, Cv=cv\}.$$
It is clear that $M_{-i}$ is spanned by the elements $L_{-i_1}L_{-i_2}\cdots L_{-i_m}v_0$ with
$i_1+\cdots+i_m=i$.
Any nonzero weight vector $v\in M_{-i}, i\in\N$ is called a \textbf{singular vector} if $L_iv=0$ for all $i\in\N$. Any nonzero weight vector $v\in M_{-i}, i\in\N$ is called an \textbf{$n$-singular vector} if $L_{ni}v=0$ for all $i\in\N$.

Now we can determine the $1$-differential module structure for the Verma module $M(h,c)$.

\begin{theorem}\label{highest weight} Suppose that $M=M(h,c)$ is a Verma module of highest weight $(h,c)\in\C^2$ defined as above and $\d: M\rightarrow M$ is a linear map.
Then $(M,\delta)$ is a $1$-differential module over $(\Vir,d_{n,a})$
if and only if $n>0$, $(n-1)c=0$, $(1-n)h\in\Z_+$ and there exists an $n$-singular vector $u\in M_{(n-1)h}$ such that
$$\d(L_{i_1}\cdots L_{i_m}v_0)=\frac{a^{i_1+\cdots+i_m}}{n^m}L_{ni_1}\cdots L_{ni_m}u-L_{i_1}\cdots L_{i_m}v_0,$$
for all $i_1\leq i_2\cdots\leq i_m\leq -1$.
In particular, if $n=1$, there exists $\xi\in\C$ such that $\d(v)=(\xi a^{-i}-1)v,\ \forall\ v\in M_{-i}.$
\end{theorem}

\begin{proof}
As in Lemma \ref{1-diff-mod}, we set $\D=\d+\id_V$ and $D=d_{n,a}+\id_{\Vir}$ for some $n\in\Z^*$ and $a\in\C^*$. We note that
$$D(L_i)=\varphi_n\tau_a(L_i)=\frac{a^i}{n}(L_{ni}-\frac{n^2-1}{24}\delta_{i,0}C),\ \ D(C)=nC.$$
Without loss of generality, we may assume that $\D\neq 0$.
Then we have $\Delta(L_iv)=D(L_i)\Delta(v)$ for any $v\in M$ and $i\in\Z$, or, more explicitly,
\begin{equation}\label{highest weight-0}
c\D(v)=\D(Cv)=D(C)\D(v)=nc\D(v),\ \forall\ v\in M
\end{equation}
and \begin{equation}\label{highest weight-1}
\D(L_iv)=\frac{a^i}{n}(L_{ni}-\frac{n^2-1}{24}\delta_{i,0}C)\D(v),\ \forall\ v\in M, i\in\Z.
\end{equation}
By Eq.~\eqref{highest weight-0}, we see $(n-1)c=0$ and Eq.~\eqref{highest weight-1} is
always equivalent to
\begin{equation}\label{highest weight-2}
\D(L_iv)=\frac{a^i}{n}L_{ni}\D(v),\ \forall\ v\in M, i\in\Z.
\end{equation}
It is easy to see that $\D(v_0)=0$ implies that $\D(M)=0$. Hence $\D(v_0)\neq0$.

Taking $i=0$ and $v=v_0$ in Eq.~\eqref{highest weight-2}, we have $L_0\D(v_0)=nh\D(v_0)$, i.e., $\D(v_0)\in M_{(n-1)h}$,
which is nonzero. Hence $-(n-1)h\in \Z_+$.
Suppose that $\D(v_0)=u\in M_{(n-1)h}\setminus\{0\}$, then by Eq.~\eqref{highest weight-2}, we can obtain
\begin{equation}\label{highest weight-2.5}
\D(L_{i_1}L_{i_2}\cdots L_{i_m}v_0)=\frac{a^{i_1+i_2+\cdots+i_m}}{n^m}L_{ni_1}L_{ni_2}\cdots L_{ni_m}u,\ \forall\ i_1,\cdots,i_m\in\Z.
\end{equation}
In particular, $L_iv_0=0$ for $i\in\N$ implies $L_{ni}u=0$ for $i\in\N$, that is, $u$ is an $n$-singular vector. On the other hand, since $M$ is a highest weight module,
we have $L_ju=0$ for  $j>(1-n)h$. If $n<0$, then the subalgebra of $\Vir$ generated by $L_{ni}, L_j, i\in\N, j>(1-n)h$ is the whole $\Vir$.
That is, $\C u$ is a trivial $\Vir$-module. By the representation theory of the Virasoro algebra (c.f. \cite{IK, KR}), we obtain $h=c=0$ in this case.
In particular, $u$ is a nonzero multiple of $v_0$ and $\C v_0$ is a trivial submodule of $M$, a contradiction.
So we must have $n>0$.

Now we can easily check that Eq.~\eqref{highest weight-2.5} indeed defines a linear map $\D$ on $V$ and $(V,\delta)$ with $\delta=\Delta-\id_V$ is a $1$-differential $(\LL,d_{n,a})$-module, provided the assumptions in the theorem are satisfied. Finally, noticing that, when $n=1$, $u\in M_0=\C v_0$ and hence
$u=\xi v_0$ for some nonzero $\xi\in\C$. It follows that $\D(v)=a^{-i}\xi v$ for all $v\in M_{-i}, i\in\Z_+$, as desired.
\end{proof}

\subsubsection{Differential structures on modules of intermediate series}
\label{sss:inter}

We next consider another class of weight Virasoro modules, or weight $\Vir$-modules in short, the modules of intermediate series. For any $\a,\b\in\C$, the module $V(\a,\b)$ of intermediate series has a
basis $\{v_i\ |\ i\in\Z\}$ subject to the following $\Vir$-module action
$$L_iv_j=(\a+j+\b i)v_{i+j},\ C v_j=0,\ \ \forall\ i,j\in\Z.$$
We can determine the differential module structures on $V(\a,\b)$ explicitly.

\begin{theorem}\label{intermediate} Let $\d: V(\a,\b)\rightarrow V(\a,\b)$ be a linear map. Then $(V(\a,\b),\d)$ is a $1$-differential module over $(\Vir,d_{n,a})$ if and only if $(n-1)\a\in\Z$ and $\d$ is defined by $\d(v_i)=\xi a^iv_{(n-1)\a+ni}-v_i$ for some $\xi\in\C$.
\end{theorem}

\begin{proof}
As before we set $\D=\d+\id_V\neq0$. Noticing $CV(\a,\b)=0$, by $\D(L_iv_j)=D(L_i)\D(v_j)$ we have
\begin{equation}\label{intermediate-1}
(\a+j+i\b)\Delta(v_{i+j})=\frac{a^i}{n}L_{ni}\Delta(v_j).
\end{equation}
When $i=0$, it becomes $n(\a+j)\Delta(v_j)=L_0\Delta(v_j)$. We observe that $\D(v_j)\neq 0$ implies that $\D(v_j)$ is a weight vector of weight $n(\a+j)$, forcing $n(\a+j)\in \a+\Z$, or equivalently, $(n-1)\a\in\Z$.

Since $\D\neq 0$, we have $\D(v_k)\neq0$ for some $k\in\Z$.
Then we have $(n-1)\a\in\Z$ and set $\D(v_k)=\xi a^kv_{(n-1)\a+nk}$ for some nonzero $\xi\in\C$.
Taking $j=k$ and replacing $i$ with $i-k$ in Eq.~\eqref{intermediate-1},
we can easily deduce
$\D(v_{i})=\xi a^{i}v_{(n-1)\a+ni}$ provided $\a+k+(i-k)\b\neq 0$.
If $\b\neq0$ and $\b\neq1$, there is at most one $i\in\Z$ such that $\a+k+(i-k)\b=0$ and we see
$\D(v_{i})=\xi a^{i}v_{(n-1)\a+ni}$ for all but at most one $i\in\Z$. If there is indeed some $i_0$ with $\a+k+(i_0-k)\b=0$, replacing $k$ by
some $k'$ with $k'\neq i_0$ and $\a+k'+(i_0-k')\b\neq 0$ in the above argument, we deduce
$\D(v_{i_0})=\xi a^{i_0}v_{(n-1)\a+ni_0}$.

Now suppose $\b=0$ or $\b=1$. Replacing $i$ with $k-i$ and $j$ with $i$ in Eq.~\eqref{intermediate-1},
we deduce
$$L_{n(k-i)}\Delta(v_i)=\xi a^{i}n(\a+i+(k-i)\b)v_{(n-1)\a+nk}.$$
Noticing that $\a+i+(k-i)\b\neq0$ implies $\D(v_i)\neq0$ and, as we have remarked in the first paragraph,
$\D(v_i)$ must be a nonzero multiple of $v_{(n-1)\a+ni}$. Comparing the coefficients in the above equation, we see $\D(v_i)=\xi a^iv_{(n-1)\a+ni}$.
We obtain $\D(v_i)=\xi a^iv_{(n-1)\a+ni}$ for all $i\in\Z$ with $\a+k+(i-k)\b\neq 0$
or $\a+i+(k-i)\b\neq0$. Since $\b=0$ or $1$, we obtain $\D(v_i)=\xi a^iv_{(n-1)\a+ni}$ for all $i\in\Z$ in this case.

Finally, it is direct to check that $\d=\D-\id_{V(\a,\b)}$, with $\D$ defined as above, indeed defines a
$1$-differential module structure on $V(\a,\b)$.
\end{proof}

\subsubsection{Differential structures on non-weight modules I}
\label{sss:nonw}

Next we consider the differential module structure for some
non-weight modules with explicit module structures, including
fraction modules and highest-weight-like modules and some modules
coming from irreducible Weyl algebra modules, which are introduced
and studied in \cite{GLZ, LGZ, LZ, TZ2}. Let $\C[t]$ be the vector
space of all polynomials in $t$ and $\mu\in\C$. We define a module
structure on $\C[t]$ by
$$L_i(t^j)=\mu^i(t-ib)(t-i)^j, C t^j=0, \ \forall\ i\in\Z, j\in\Z_+.$$
We denote this module by $\Omega(\mu,b)$.

\begin{theorem}\label{Omega} Let $\d: \Omega(\mu,b)\rightarrow \Omega(\mu,b)$ be a linear map. Then $(\Omega(\mu,b), \d)$ is a
$1$-differential module over $(\Vir, d_{n,a})$ if and only if $a\mu^{n-1}=1$ and $\d$ is defined by $\d(t^j)=\xi\big(\frac{t}{n}\big)^j-t^j$ for some $\xi\in\C$.
\end{theorem}

\begin{proof} As before, we set $\D=\d+\id_{\Omega(\mu,b)}$ and $D=d_{n,a}+\id_{\Vir}$. Without loss of generality, we assume that $\D\neq 0$.
By Lemma \ref{1-diff-mod},
we get
\begin{equation}\label{omega-1}
\Delta(L_i(t^j))=\Delta(\mu^i(t-ib)(t-i)^j)=\frac{a^i}{n}L_{ni}\Delta(t^j),\ i\in\Z, j\in\Z_+.
\end{equation}
Taking $i=0$, we have $\Delta(t^{j+1})=\frac{1}{n}L_0\Delta(t^j)=\frac{t}{n}\Delta(t^j)$.
By an induction on $j$, we obtain $\Delta(t^j)=(\frac{t}{n})^j\Delta(1)$ for all $j\in\Z_+$.

Suppose $\D(1)=h(t)\in\C[t]$ and substitute $\Delta(t^j)=(\frac{t}{n})^jh(t)$ into Eq.~\eqref{omega-1}, we get
$$\mu^i(\frac{t}{n}-ib)(\frac{t}{n}-i)^jh(t)=\frac{\mu^{ni}a^i}{n}(t-nib)(\frac{t-ni}{n})^jh(t-ni).$$
So we have
$$h(t)=a^i\mu^{(n-1)i}h(t-ni),\ \forall\ i\in\Z,$$
which implies that $a\mu^{n-1}=1$ and $h=\xi\in\C$. This completes the proof.
\end{proof}

\subsubsection{Differential structures on non-weight modules II}
\label{sss:nonw2}

Let $\C(t)$ be the fraction field of $\C[t]$
and $\C(t)[s]$ be the skew polynomial ring with $s
t^i=t^i(s +i)$. Then $\mathcal{K}=\C[t^{\pm1}, s]$ is a subalgebra of $\C(t)[s]$. For any $\a\in\C(t)$, we see that
$s-\a$ is an irreducible polynomial in $\C(t)[s]$ and hence
$$A=\mathcal{K}/\left(\mathcal{K}\cap\big(\C(t)[s](s-\a)\big)\right)$$
is an irreducible $\mathcal{K}$-module. For any $\b\in\C$, we
can define a $\Vir$-module structure on $A$ by
\begin{equation} L_if(t)=\big(\p+\a+i\b\big)t^if(t),\end{equation}
where $\p=t\frac{d}{dt}$ and we still use $f(t)$ to denote its image in $A$. We denote this
$\Vir$-module by $A_{\a,\b}$.

Note that when $\a\in\C$, the modules $A_{\a,\b}$ are just those modules considered in Theorem \ref{intermediate}.
So we always assume that $\a\in\C(t)\setminus\C$ in the following.

Suppose $\a=\a_1(t)/\a_2(t)$ such that $\a_1, \a_2\in\C[t]$ are coprime and
$\a_2(t)=(t-a_1)^{l_1}\cdots(t-a_r)^{l_r}$ for some distinct $a_1,\cdots,a_r\in\C^*$ and $l_1,\cdots,l_r\in\N$. Then as a vector space,
$A_{\a,\b}$ may be regarded as the associative subalgebra of $\C(t)$ generated by $t$ and $(t-a_i)^{-1}, i=0,1,\cdots,r$,
where we have taken $a_0=0$. In what follows, we will fix these notations and make the identification
\begin{equation}
A_{\a,\b}=\C[t^{\pm1}, (t-a_1)^{-1},\cdots,(t-a_r)^{-1}].
\label{eq:aab}
\end{equation}
It is clear that $A_{\a,\b}$ has a basis consisting of the following elements:
$$1, t^k, t^{-k}, (t-a_i)^{-k}, k\in\N, i=1,\cdots,r.$$

To describe the differential module
structure on $A_{\a,\b}$, we need the following preliminary lemmas:

\begin{lemma}\label{partial} Let $f(t), g(t)\in A_{\a,\b}$. If $\p(f)=gf$, then $f=c t^{m_0}(t-a_1)^{m_1}\cdots (t-a_r)^{m_r}$ for some $c\in\C$ and $m_i\in\Z, i=0,1,\cdots,r$.
\end{lemma}

\begin{proof}
Assume that
$f(t)=\frac{f_1(t)}{\a_2^m(t)}$
and $g(t)=\frac{g_1(t)}{\a_2^k(t)}$ for suitable $f_1, g_1\in\C[t^{\pm1}]$ and $m,k\in\Z_+$, then $\p(f)=gf$ becomes
$$\a_2^{k}(\p(f_1)\a_2-mf_1\p(\a_2))=\a_2f_1g_1.$$
Regarding elements in $\C(t)$ as $\C$-valued functions, we can deduce
$$f_1(t)=\pm\exp\Big(\int \frac{\a_2g_1+m\a_2^k\p(\a_2)}{t\a_2^{k+1}} \mathrm{d}t\Big).$$
To ensure that $f_1(t), g_1(t), \a_2(t)\in \C[t^{\pm1}]$,
we must have
that $\frac{\a_2g_1+m\a_2^k\p(\a_2)}{t\a_2^{k+1}}$ is a linear combination of $(t-a_i)^{-1}, i=0,1,\cdots,r$
and hence $f_1(t)=c\prod_{i=0}^r(t-a_i)^{m_i}$ for some $c\in\C$ and $m_i\in\Z_+$. The lemma follows.
\end{proof}

\begin{lemma}\label{n=1}
Let $\omega$ be a primitive root of unity of degree $d$ and $b_1,\cdots,b_s\in\C^*$ such that all elements
$b_i\omega^j, i=1,\cdots,s; j=1,\cdots,d$ are distinct. Denote $A=\C[t^{\pm1}, (t-b_i\omega^j)^{-1}, i=1,\cdots,s; j=1,\cdots,d]$.
Let $A^\omega$ be the subspace of $A$ consisting of all elements $f(t)\in A$ satisfying the condition $f(\omega t)-f(t)=0$. Then we have
$$A^\omega=\C[t^{\pm d}]\oplus \sum_{k\in\N}\sum_{i=1}^s\C f_{i,k}(t),$$
where $f_{i,k}(t)=\sum_{j=1}^{d}(\omega^{j}t-b_i)^{-k}.$
\end{lemma}

\begin{proof} It is straightforward to check that any polynomial $f\in A^{\omega}$ does satisfy the equation $f(\omega t)-f(t)=0$.
On the other hand, take any element $f\in A_{\a,\b}$ such that $f(\omega t)-f(t)=0$. Writing $f$ as the linear combination
of the basis elements $1, t^{k}, t^{-k}, (\omega^jt-b_i)^{-k}, i=1,\cdots,s; j=1,\cdots, d; k\in\N$, we can check directly that
$f$ is a linear combination of elements in $\C[t^{\pm d}]$ and $f_{i,k}$, where $i=1,\cdots,s$ and $k\in\N$.
\end{proof}

\begin{lemma}\label{n=-1}
Let $\omega\in\C^*$ and $g\in\C(t)\setminus\{0\}$. If $g(\omega t^{-1})+g(t)=0$,
then $g(t)$ must be a polynomial of the following form:
$$\frac{t^{k}(t^2-\omega)\prod_{i=1}^l(t-\l_i)(\l_i t-\omega)}{\prod_{i=1}^{m}(t-\mu_i)(\mu_i t-\omega)},$$
where $\l_i, \mu_i\in\C^*$ (possibly not distinct) and $m=l+k+1$.
\end{lemma}

\begin{proof} Suppose $g(t)=g_1(t)/g_2(t)$ for some nonzero polynomials $g_1(t), g_2(t)\in\C[t]$ which do not
have common divisors other than $t$ and $t\pm \sqrt{\omega}$, where $\sqrt{\omega}$ is a fixed square root of $\omega$.
Then we have
\begin{equation}\label{n=-1 1}
\frac{g_1(\omega t^{-1})}{g_2(\omega t^{-1})}+\frac{g_1(t)}{g_2(t)}=0\quad \text{and}
\quad \frac{g_2(\omega t^{-1})}{g_1(\omega t^{-1})}+\frac{g_2(t)}{g_1(t)}=0.
\end{equation}
For any $\l\in\C^*$ with $\lambda\neq \pm \sqrt{\omega}$, we see that $g_1(\l)=0$ if and only if $g_1(\omega \l^{-1})=0$, that is, $t-\l$ divides
$g_1(t)$ if and only if $t-\omega \l^{-1}$ divides $g_1(t)$. Similarly, $t-\l$ divides
$g_2(t)$ if and only if $t-\omega \l^{-1}$ divides $g_2(t)$.
Then, without loss of generality, we may suppose
$$g_1(t)=t^{r_0}(t-\sqrt{\omega})^{r_1}(t+\sqrt{\omega})^{r_2}\prod_{i=1}^l(t-\l_i)(\l_i t-\omega)$$
and
$$g_2(t)=t^{s_0}(t-\sqrt{\omega})^{s_1}(t+\sqrt{\omega})^{s_2}\prod_{i=1}^m(t-\mu_i)(\mu_i t-\omega),$$
where $r_i, s_i\in\Z_+, i=0,1,2$ and all $\l_i, \omega\l_i^{-1}, \mu_j, \omega\mu_j^{-1}\in\C^*, i=1,\cdots,l, j=1,\cdots,m$ are distinct.
It is easy to check that
$$\frac{g_1(\omega t^{-1})}{g_2(\omega t^{-1})}=
\frac{(-1)^{r_1}t^{-2r_0-r_1-r_2-2l}(\sqrt{\omega})^{2r_0+r_1+r_2+2l}}{(-1)^{s_1}t^{-2s_0-s_1-s_2-2m}(\sqrt{\omega})^{2s_0+s_1+s_2+2m}}\frac{g_1(t)}{g_2(t)},$$
which, by Eq.~\eqref{n=-1 1}, implies that
\begin{equation}\label{n=-1 2}
1+\frac{(-1)^{r_1}t^{-2r_0-r_1-r_2-2l}(\sqrt{\omega})^{2r_0+r_1+r_2+2l}}{(-1)^{s_1}t^{-2s_0-s_1-s_2-2m}(\sqrt{\omega})^{2s_0+s_1+s_2+2m}}=0.
\end{equation}
This is equivalent to the following conditions:
$$r_1-s_1, r_2-s_2\in2\Z+1,\quad 2r_0+r_1+r_2+2l=2s_0+s_1+s_2+2m.$$
The statement of this lemma follows easily.
\end{proof}

We next prove another preliminary result.

\begin{lemma}
Let $(A_{\a,\b},\delta)$ be a $1$-differential module over $(\Vir, d_{n,a})$.
Suppose $\D:=\delta+\id_{A_{\a,\b}}$ is nonzero. Then
$\D(t^if)-a^it^{ni}\D(f)=0$ for all $i\in\Z$ and $f\in A_{\a,\b}$.
\label{lem:delta}
\end{lemma}

\begin{proof}
By the assumptions, we have
$\D(L_if)=D(L_i)\D(f)$, that is,
\begin{equation}\label{Ab-1}
\D((\a+\p+i\b)t^if)=\frac{a^i}{n}L_{ni}\D(f)=\frac{a^i}{n}(\a+\p+ni\b)t^{ni}\D(f),\ \forall\ i\in\Z, f\in A_{\a,\b}.
\end{equation}
Taking $i=0$, we have $\D((\a+\p)f)=\frac{1}{n}(\a+\p)\D(f)$ and in particular, $\D((\a+\p)t^if)=\frac{1}{n}(\a+\p)\D(t^if)$. Substituting
it back into Eq.~\eqref{Ab-1}, we obtain
\begin{equation*}
\frac{1}{n}(\a+\p)\D(t^if)+i\b\D(t^if)=\frac{a^i}{n}(\a+\p+ni\b)t^{ni}\D(f),\ \forall\ i\in\Z, f\in A_{\a,\b}.
\end{equation*}
Simplifying it we get
\begin{equation}\label{D(tf)}
\p\big(\D(t^if)-a^it^{ni}\D(f)\big)=-(\a+ni\b)\big(\D(t^if)-a^it^{ni}\D(f)\big),\ \forall\ i\in\Z, f\in A_{\a,\b}.
\end{equation}

For any $f\in A_{\a,\b}$ and any $i,j\in\Z$, by Lemma \ref{partial}, we see that
$$\aligned
& \D(t^if)-a^it^{ni}\D(f)=c_it^{l_0-ni\b}\prod_{s=1}(t-a_s)^{l_s},\\
& \D(t^jf)-a^jt^{nj}\D(f)=c_jt^{l'_0-nj\b}\prod_{s=1}(t-a_s)^{l'_s}.
\endaligned$$
Then Eq.~\eqref{D(tf)} implies
$$-(\a+ni\b)=l_0-ni\b+\sum_{s=1}^r\frac{tl_s}{t-a_s},\quad -(\a+nj\b)=l'_0-nj\b+\sum_{s=1}^r\frac{tl'_s}{t-a_s},$$
forcing $l_s-l_s'=0$ for $s=0,\cdots,r$ and $\a=l_0+\sum_{s=1}^r\frac{tl_s}{t-a_s}$. Denote $g(t)=t^{l_0}\prod_{s=1}^r(t-a_s)^{l_s}$, which is determined by $\a$. We see
\begin{equation}\label{c}
\D(t^if)-a^it^{ni}\D(f)=c_it^{-ni\b}g(t),\quad \D(t^jf)-a^jt^{nj}\D(f)=c_jt^{-nj\b}g(t).
\end{equation}
Similarly, we get
\begin{equation}\label{c'}
\D(t^{j-i}t^if)-a^{j-i}t^{n(j-i)}\D(t^if)=c't^{-n(j-i)\b}g(t),
\end{equation}
for some $c'\in\C$. Combining Eqs.~\eqref{c} and \eqref{c'}, we deduce
$$c_jt^{-nj\b}=c't^{-n(j-i)\b}+c_ia^{j-i}t^{n(j-i)}t^{-ni\b},\ \forall\ i,j\in\Z.$$
If $\b\neq0,-1$, we find $c_j=0$ by taking $i$ sufficiently large and hence the claim follows from
Eq.~\eqref{c}.
If $\b=0$, taking $j\neq i$, we see $c_i=0$ and the claim also holds. Finally, if $\b=-1$, taking
$i\neq0$, we obtain $c'\neq0$ and Eq.~\eqref{c'} also proves the lemma.
\end{proof}

We can now give the classification of $1$-differential structures on $A_{\a,\b}$.

\begin{theorem}\label{Ab} Let $a, n, \a,\b$ and $A_{\a,\b}$ be as above.
Let $\d: A_{\a,\b}\rightarrow A_{\a,\b}$ be a linear map.
Then $(A_{\a,\b}, \d)$ is a $1$-differential $(\Vir, d_{n,a})$-module
if and only if one of the following conditions holds
\begin{enumerate}
\item $n=1$, $a$ is a primitive root of unity of degree $d$, and
         $$A_{\a,\b}=\C[t^{\pm1}, (t-a_ia^j)^{-1}, i=1,\cdots,s; j=1,\cdots,d]$$
    for some $a_i$ such that all $a_{i}a^j$ are distinct, and $\a(t)$ is a sum of
        $$\a_0(t)=\sum_{i=1}^s\sum_{j=1}^{d}a_{i}(m_{i0}+m_{i1}+\cdots+m_{i,j-1})(a^{-j}t-a_i)^{-1}$$
    where $m_{ij}\in\Z$ with $\sum_{j=0}^{d-1}m_{ij}=0$ for all $i=1,\cdots,s$, and is an element in $A^\omega$ as defined in Lemma \ref{n=1} (taking $A=A_{\a,\b}$ and $\omega=a$). In this case, we have
        $$\delta(f)(t)=cf(at)\prod_{i=1}^{s}\prod_{j=0}^{d-1}(t-a_{i}a^j)^{m_{ij}}-f(t),\ \forall\ f\in A_{\a,\b},$$
    for some $c\in\C^*$.
\item $n=-1$ and $A_{\a,\b}=\C[t^{\pm1}, (t-a_i)^{-1}, (t-a_i^{-1}a)^{-1}, i=1,\cdots,s]$ for some $a_i$ such that
      all $a_i$ are distinct (maybe $a_i=a_j^{-1}a$), and $\a(t)$ is a sum of
        $$\a_0(t)=-\frac{m_0}{2}-\frac{m_i}{2}\sum_{i=1}^s\frac{(a_i^2-a)t}{(t-a_i)(a_it-a)},$$
     where $m_i\in\Z$, and is an elements in $A_{\a,\b}$ of the form
         $$b\frac{t^{k}(t^2-a)\prod_{i=1}^l(t-\l_i)(\l_i t-a)}{\prod_{i=1}^{m}(t-\mu_i)(\mu_i t-a)},\ b\in\C,$$
     where $k\in\Z, m,l\in\N$ with $m=l+k+1$, and $\l_i, \mu_i\in\C^*$ with each $\mu_i=a_j$ or $\mu_i=a_j^{-1}a$ for some $1\leq j\leq s$.
     In this  case, we have
        $$\delta(f)(t)=cf(at^{-1})t^{m_0} \prod_{i=1}^{s}(t-a_{i})^{m_{i}}(t-a_i^{-1}a)^{-m_i}-f(t),\ \forall\ f\in A_{\a,\b},$$
            for some $c\in\C^*$.
     \end{enumerate}
\end{theorem}

\begin{proof} Let $(A_{\a,\b},\delta)$ be a $1$-differential module over $(\Vir, d_{n,a})$.
Denote $\D=\delta+\id_{A_{\a,\b}}$ and suppose $\D\neq0$ as before.

By Lemma~\ref{lem:delta}, we have $\D(t^if)-a^it^{ni}\D(f)=0$, that is,
$\D(t^if)=a^it^{ni}\D(f)$ for all $f\in A_{\a,\b}$.
By linearity we easily deduce
\begin{equation}\label{Ab-2}
\D(gf)=g(at^n)\D(f),\ \forall\ g\in\C[t^{\pm1}], f\in A_{\a,\b}.
\end{equation}
Moreover, for $g_1(t), g_2(t)\in\C[t], f(t)\in A_{\a,\b}$ with $g_2(t)\neq0$ and $g_1(t)f(t)/g_2(t)\in A_{\a,\b}$, we have
$$\Delta(g_1(t)f(t))=g_2(at^n)\Delta\Big(\frac{g_1(t)}{g_2(t)}f(t)\Big)=g_1(at^n)\Delta(f(t)),$$
yielding $\Delta(g_1(t)f(t)/g_2(t))=g_1(at^n)\Delta(f(t))/g_2(at^n)$. That is, Eq.~\eqref{Ab-2} also holds for
$g(t)\in\C(t), f(t)\in A_{\a,\b}$ provided $gf\in A_{\a,\b}$.
As a result, we obtain $\D(f)=f(at^n)\D(1)$ for all $f\in A_{\a,\b}$.
In particular, we have $\D(\a)=\a(at^n)\D(1)$.

For convenience, denote $h(t):=\D(1)\in A_{\a,\b}$. Note
\begin{equation*}\p(\D(f)) = t\frac{d}{dt}(f(at^n)h) = t\frac{df}{du}\big|_{u=at^n}\cdot ant^{n-1}h+tf(at^n)\frac{dh}{dt}
 =n\D(\p(f))+f(at^n)\p(h)
\end{equation*}
for all $f\in A_{\a,\b}$.
Then Eq.~\eqref{Ab-1} is equivalent to
$$\p(h)=(n\a(at^n)-\a(t))h.$$

Recall $A_{\a,\b}=\C[t^{\pm1},(t-a_i)^{-1}, i=1,\cdots,r]$ from Eq.~\eqref{eq:aab}. By Lemma \ref{partial}, we see that $h(t)=ct^{m_0}\prod_{i=1}^{r}(t-a_{i})^{m_{i}}$ for some $c\in\C^*$ and $m_{i}\in\Z$. In particular, $\D(1)=h(t)$ is invertible in $A_{\a,\b}$. We will distinguish two cases.

\smallskip
\noindent\textbf{Case 1.} $r=0$.
\smallskip

We have $A_{\a,\b}=\C[t^{\pm1}]$, $h(t)=ct^j$ for some $c\in\C^*$ and $j\in\Z$ and Eq.~\eqref{Ab-1} is equivalent to
$$n\a(at^n)-\a(t)=j,$$
which can hold only when $n=\pm1$ since $\a\in\C[t^{\pm1}]\setminus\C$. If $n=1$, we have that $a$ is a primitive root of unity
of degree $d\in\N$ and $\a(t)\in\C[t^{\pm d}]$. Moreover, we have $j=0$ and $h(t)=c$.
If $n=-1$, we can deduce that $\a(t)=\frac{j}{2}+\sum_{i=0}^p(b_it^i-b_ia^{i}t^{-i}).$

\smallskip
\noindent\textbf{Case 2.} $r\geq 1$.
\smallskip

Now we suppose that $r\geq 1$.
Then $\D((t-a_i)^{-1})=(at^n-a_i)^{-1}h(t)\in A_{\a,\b}$ implies that $(at^n-a_i)^{-1}\in A_{\a,\b}$
since $h(t)$ is invertible in $A_{\a,\b}$.
By induction, we deduce
$$(a^{1+n+\cdots+n^{k-1}}t^{n^k}-a_i)^{-1}\in A_{\a,\b},\ \forall\ k\in\N, i=1,2,\cdots,r,$$
which implies that the $n^{2k}$-th power roots of $a_i/a^{1+n+\cdots+n^{k-1}}$
all lies in $\{a_1,a_2,\cdots,a_r\}$. This can occur only when $n=\pm1$.

\smallskip
\noindent\textbf{Subcase 2.1.} $r\geq 1$ and $n=1$.
\smallskip

We have $(a^kt-a_i)^{-1}\in\ A_{\a,\b}$ for all $k\in\N$, that is, all $a_i/a^k$ lies in $\{a_1,\cdots,a_r\}$.
Then $a^k=1$ for some $k\in\N$. Let $d$ be the smallest such $k$. It is easy to see that $r=sd$ for some $s\in\N$ and by reordering these $a_1,\cdots,a_r$, we may assume
$$\{a_1,a_2,\cdots,a_r\}=\{a_ia^j\ |\ i=1,\cdots,s; j=1,\cdots, d\}.$$
Now suppose that $h(t)=ct^{m_0}\prod_{i=1}^{s}\prod_{j=1}^{d}(t-a_{i}a^j)^{m_{ij}}$ for some $c\in\C$ and $m_{ij}\in\Z$. Then we have
$$\a(at)-\a(t)=\frac{\p(h)}{h}=m_0+\sum_{i=1}^s\sum_{j=1}^{d}\frac{m_{ij}t}{t-a_ia^j}=m_0+\sum_{i=1}^s\sum_{j=1}^{d}m_{ij}+\sum_{i=1}^s\sum_{j=1}^{d}\frac{m_{ij}a_{i}a^{j}}{t-a_{i}a^{j}}$$
Write $\a(t)$ as a combination of the basis elements $1, t^k, t^{-k}, (t-a_ia^j)^{-k}, k\in\N, i=1,\cdots,s; j=1,\cdots, d$. By easy computations,
we can deduce
\begin{equation}\label{alpha}\begin{split}
\a(t)= & \sum_{k\in d\Z}b_kt^{k}+\sum_{i=1}^s\sum_{j=1}^{d}\big(b'_{i}+a_{i}(m_{i0}+m_{i1}+\cdots+m_{i,j-1})\big)(a^{-j}t-a_i)^{-1}\\
       & +\sum_{k\geq 2}\sum_{i=1}^s\sum_{j=1}^{d}b''_{k,i}(a^{-j}t-a_i)^{-k},
\end{split}\end{equation}
where $m_{i0}=m_{id}$ and only finitely many of $b_k, b'_i, b''_{i,k}\in\C$ are nonzero.
Moreover, we have $m_{0}=\sum_{j=1}^{d}m_{ij}=0$ for all $i=1,\cdots,s$ and hence
$$h(t)=c\prod_{i=1}^{s}\prod_{j=1}^{d}(t-a_{i}a^j)^{m_{ij}}.$$
If we denote
\begin{equation*}\begin{split}
\a_0(t)= \sum_{i=1}^s\sum_{j=1}^{d}a_{i}(m_{id}+m_{i1}+\cdots+m_{i,j-1})(a^{-j}t-a_i)^{-1},\\
\end{split}\end{equation*}
then it is straightforward to check that $\a_0(at)-\a_0(t)=\p(h)/h$.
Hence $(\a-\a_0)(at)-(\a-\a_0)(t)=0$, i.e., $\a-\a_0\in A^{\omega}$ as defined in Lemma \ref{n=1} (taking $A=A_{\a,\b}$ and $\omega=a$), and $\a$ is of the form
as in Eq.~\eqref{alpha}. The result follows in this case.

\smallskip
\noindent\textbf{Subcase 2.2.} $r\geq 1$ and $n=-1$.
\smallskip

If $n=-1$, then we have $(at^{-1}-a_i)^{-1}=-t(a_it-a)^{-1}\in\ A_{\a,\b}$, forcing $a_i^{-1}a=a_j$ for some $1\leq j\leq r$.
Similarly, reordering these $a_1,\cdots,a_r$, we may assume
$$\{a_1,a_2,\cdots,a_r\}=\{a_i, a_i^{-1}a\ |\ i=1,\cdots,s\}.$$
Now suppose that $h(t)=ct^{m_0}\prod_{i=1}^{s}(t-a_i)^{m_i}(t-a_i^{-1}a)^{m'_{i}}$ for some $c\in\C^*$ and $m_{i},m'_i\in\Z$.
Then we have
$$-\a(at^{-1})-\a(t)=\frac{\p(h)}{h}=\frac{th'(t)}{h(t)}=t(\ln(h(t)))',$$
where $h'(t)$ and $(\ln(h(t)))'$ refer to the derivative of $h(t)$ with respect to $t$.
Replacing $t$ by $at^{-1}$, we get
$$-\a(t)-\a(at^{-1})=at^{-1}\frac{h'(at^{-1})}{h(at^{-1})}=t\frac{at^{-2}h'(at^{-1})}{h(at^{-1})}=-t(\ln(h(at^{-1})))'.$$
This indicates that $\ln(h(t))+\ln(h(at^{-1}))$ is a constant, or equivalently, $h(t)h(at^{-1})$ is a nonzero constant. More precisely,
we see that
$$t^{m_0}\prod_{i=1}^{s}(t-a_i)^{m_i}(t-a_i^{-1}a)^{m'_{i}}t^{-m_0}\prod_{i=1}^{s}(at^{-1}-a_i)^{m_i}(at^{-1}-a_i^{-1}a)^{m'_{i}}$$
is a nonzero constant, forcing $m_i+m'_i=0$ for all $i=1,\cdots,s$.
Denoting
$$\a_0(t)=-\frac{m_0}{2}-\frac{m_i}{2}\sum_{i=1}^s\frac{(a_i^2-a)t}{(t-a_i)(a_it-a)},$$
we have
$$-\a_0(at^{-1})-\a_0(t)=\frac{\p(h)}{h}=m_0+\sum_{i=1}^s\left(\frac{m_{i}t}{t-a_{i}}-\frac{m_it}{t-a_i^{-1}a}\right).$$
Hence $(\a-\a_0)(at^{-1})+(\a-\a_0)(t)=0$. By Lemma \ref{n=-1}, we can give a description of $\a(t)$, that is, $\a(t)$ can be written as the sum of $\a_0(t)$ and  an element in $A_{\a,\b}$ of the following form:
$$b\frac{t^{k}(t^2-a)\prod_{i=1}^l(t-\l_i)(\l_i t-a)}{\prod_{i=1}^{m}(t-\mu_i)(\mu_i t-a)},\ b\in\C,$$
where $k\in\Z, l, m\in\N$ with $m=k+l+1$
and all $\l_i, \mu_i\in\C^*$ (possibly not distinct) with $\mu_i, \mu_i^{-1}a\in\{a_1,\cdots,a_r\}$.
This is what we need.

It is straightforward to check that the operators thus obtained give differential modules.
\end{proof}

\noindent
{\bf Acknowledgements:} This research is supported by the National Natural Science Foundation of China (Grant No. 11471294 and 11771190), the China Scholarship Council (No. 201606180088) and the Outstanding Young Talent Research Fund of Zhengzhou University (Grant No. 1421315071). The authors thank the referee for helpful suggestions.

\end{document}